\documentclass[11pt]{amsart}
\usepackage[a4paper,margin=30mm]{geometry}
\usepackage{amsmath,amssymb,amsthm,mathtools}
\usepackage{microtype}
\usepackage{enumitem}
\usepackage{hyperref}

\newtheorem{theorem}{Theorem}[section]
\newtheorem{lemma}[theorem]{Lemma}

\theoremstyle{remark}

\newcommand{\cyc}{\sum_{\mathrm{cyc}}}

\title{A Proof of Liu's Conjecture on the Fundamental Triangle Inequality}
\author{Tserendorj Batbold}
\address{
Department of Mathematics\\
National University of
Mongolia\\
Ulaanbaatar, Mongolia}
\email{tsbatbold@hotmail.com}

\begin{document}

\maketitle

\begin{abstract}
Let $a,b,c$ be the side lengths of a triangle, and let $R$ and $r$
denote its circumradius and inradius, respectively. Liu proposed the
conjecture
\[
 \cyc \left(\frac{a(b+c-a)}{bc}\right)^k
 \ge 2+\left(\frac{2r}{R}\right)^k,\qquad k>1,
\]
with the reverse inequality for $k<1$. We prove this conjecture by
reducing it to an algebraic inequality for three positive variables
with prescribed sum and product. We also determine the equality cases.
\end{abstract}

\section{Introduction}

The fundamental triangle inequality states that
\[
2R^2+10Rr-r^2-2(R-2r)\sqrt{R^2-2Rr}
\le s^2
\le
2R^2+10Rr-r^2+2(R-2r)\sqrt{R^2-2Rr},
\]
where $s$, $R$, and $r$ denote the semiperimeter, circumradius, and
inradius of a nondegenerate triangle $ABC$, respectively. Equivalently,
\begin{equation}\label{eq:fundamental}
 s^4-(4R^2+20Rr-2r^2)s^2+r(4R+r)^3\le0.
\end{equation}
This is a classical relation among the basic metric elements of a
triangle; see, for example, Mitrinovi\'c, Pe\v{c}ari\'c and Volenec
\cite{Mitrinovic1989}.

Let $a,b,c$ be the side lengths of $ABC$. In \cite{Liu2022}, Liu
obtained
\begin{equation}\label{eq:k2}
 \cyc \left(\frac{a(b+c-a)}{bc}\right)^2
 \ge 2+\left(\frac{2r}{R}\right)^2
\end{equation}
and showed that \eqref{eq:k2} is equivalent to
\eqref{eq:fundamental}. He also obtained
\begin{equation}\label{eq:kminus2}
 \cyc \left(\frac{a(b+c-a)}{bc}\right)^{-2}
 \le 2+\left(\frac{2r}{R}\right)^{-2},
\end{equation}
which is also equivalent to \eqref{eq:fundamental}. Thus
\eqref{eq:k2} and \eqref{eq:kminus2} are the cases $k=2$ and $k=-2$,
respectively, of the same family of inequalities. Liu formulated their
generalization as Conjecture~5.1 in \cite{Liu2022}: for every real
$k>1$,
\begin{equation}\label{eq:liu}
 \cyc \left(\frac{a(b+c-a)}{bc}\right)^k
 \ge 2+\left(\frac{2r}{R}\right)^k,
\end{equation}
with the reverse inequality for $k<1$. For $k=1$, the corresponding
relation is the identity
\begin{equation}\label{eq:k1}
 \cyc \frac{a(b+c-a)}{bc}
 =2+\frac{2r}{R}.
\end{equation}

The purpose of this paper is to prove this conjecture. After the Ravi
substitution, the problem reduces to an algebraic inequality for three
positive variables with prescribed sum and product. This reduction also
determines the equality cases.

\section{Proof of the conjecture}
We begin with the following auxiliary result.

\begin{lemma}\label{lem:power}
Let $0<q\le 1$ and let $x,y,z>0$ satisfy
\begin{equation}\label{eq:constraints}
 x+y+z=2+q,\qquad xyz=q.
\end{equation}
Then
\begin{equation}\label{eq:power}
 x^k+y^k+z^k
 \begin{cases}
 \ge 2+q^k, & k>1,\\[2mm]
 \le 2+q^k, & k<1.
 \end{cases}
\end{equation}
If $k\ne0,1$, equality holds if and only if $(x,y,z)$ is a permutation of
$(q,1,1)$. For $k=0$ and $k=1$, equality holds for every
$(x,y,z)$ satisfying \eqref{eq:constraints}.
\end{lemma}

\begin{proof}
For $k=0$ and $k=1$, the assertion follows immediately from
\eqref{eq:constraints}. Hence assume $k\ne0,1$.

If $q=1$, then $x+y+z=3$ and $xyz=1$. Equality in AM--GM gives
$x=y=z=1$, and the assertion follows. Hence assume $0<q<1$.

Consider
\[
 \mathcal S_q=\{(x,y,z)\in(0,\infty)^3:x+y+z=2+q,\ xyz=q\}.
\]
This set is nonempty, since $(q,1,1)\in\mathcal S_q$, and it is compact. Indeed, every coordinate is at most $2+q$. Moreover,
\[
 yz\le \left(\frac{y+z}{2}\right)^2\le\frac{(2+q)^2}{4},
\]
so
\[
 x=\frac{q}{yz}\ge\frac{4q}{(2+q)^2},
\]
and cyclically the same positive lower bound holds for $y$ and $z$. Hence $\mathcal S_q$ is a closed subset of a compact box. Thus the continuous function
\[
 F(x,y,z)=x^k+y^k+z^k
\]
attains its minimum and maximum on $\mathcal S_q$.

We next verify that the constraint gradients are independent on $\mathcal S_q$. If
\[
 (1,1,1)\quad\text{and}\quad(yz,zx,xy)
\]
were linearly dependent, positivity would imply $x=y=z$. Then \eqref{eq:constraints} would yield
\[
 \left(\frac{2+q}{3}\right)^3=q,
\]
or
\[
 (2+q)^3-27q=(q-1)^2(q+8)=0,
\]
contrary to $0<q<1$. Hence the Lagrange multiplier conditions apply at every extremum of $F$ on $\mathcal S_q$.

At an extremum there exist real $\lambda,\mu$ such that
\[
 kx^{k-1}=\lambda+\mu yz,
\]
and cyclically. Since $xyz=q$, multiplication by the corresponding variable gives
\[
 kx^k-\lambda x-\mu q=0,
\]
with the corresponding equations for $y$ and $z$. Hence $x,y,z$ are positive zeros of
\[
 \phi(t)=kt^k-\lambda t-\mu q.
\]
For $t>0$,
\[
 \phi''(t)=k^2(k-1)t^{k-2},
\]
so $\phi'$ is strictly monotone. Hence $\phi$ has at most two positive
zeros: three distinct positive zeros would, by Rolle's theorem, give two
distinct zeros of $\phi'$, which is impossible. Therefore at least two
of $x,y,z$ are equal at every extremum.

By symmetry, write $y=z=t$. From \eqref{eq:constraints},
\[
 x=\frac{q}{t^2},\qquad \frac{q}{t^2}+2t=2+q.
\]
Hence
\[
 2t^3-(2+q)t^2+q=0,
\]
which factors as
\begin{equation}\label{eq:factor}
 (t-1)(2t^2-qt-q)=0.
\end{equation}

If $t=1$, then $(x,y,z)=(q,1,1)$ and
\begin{equation}\label{eq:F1}
 F_1=2+q^k.
\end{equation}

For the other positive solution of \eqref{eq:factor},
\[
 q=\frac{2t^2}{1+t},\qquad x=\frac{2}{1+t}.
\]
Since $0<q<1$, this positive solution satisfies $0<t<1$. The corresponding value of $F$ is
\[
 F_2=\left(\frac{2}{1+t}\right)^k+2t^k.
\]
Using $q=2t^2/(1+t)$, we obtain
\begin{align*}
 F_2-F_1
 &=\left(\frac{2}{1+t}\right)^k+2t^k-2
   -\left(\frac{2t^2}{1+t}\right)^k\\
 &=(1-t^k)\left[\left(\frac{2}{1+t}\right)^k(1+t^k)-2\right].
\end{align*}
For $k>1$, strict convexity of $u\mapsto u^k$ gives
\[
 \frac{1+t^k}{2}>\left(\frac{1+t}{2}\right)^k.
\]
Since $1-t^k>0$, it follows that $F_2>F_1$. Thus the global minimum is
$F_1$, proving the first inequality in \eqref{eq:power}.

If $0<k<1$, strict concavity gives
\[
 \frac{1+t^k}{2}<\left(\frac{1+t}{2}\right)^k.
\]
Since $1-t^k>0$, we obtain $F_2<F_1$.

If $k<0$, the function $u\mapsto u^k$ is strictly convex, and hence
\[
 \frac{1+t^k}{2}>\left(\frac{1+t}{2}\right)^k.
\]
In this case $1-t^k<0$, so again $F_2<F_1$. Therefore $F_1$ is the
global maximum for every $k<1$, $k\ne0$, and the second inequality in
\eqref{eq:power} follows.

The strictness of the above comparisons for $0<t<1$ shows that, for
$k\ne0,1$, equality can occur only when $(x,y,z)$ is a permutation of
$(q,1,1)$. Conversely, every permutation of $(q,1,1)$ gives equality.
\end{proof}

\begin{theorem}\label{thm:main}
Let $ABC$ be a nondegenerate triangle with side lengths $a,b,c$, circumradius $R$, and inradius $r$. If $k>1$, then
\begin{equation}\label{eq:mainplus}
 \cyc \left(\frac{a(b+c-a)}{bc}\right)^k
 \ge 2+\left(\frac{2r}{R}\right)^k.
\end{equation}
If $k<1$, then
\begin{equation}\label{eq:mainminus}
 \cyc \left(\frac{a(b+c-a)}{bc}\right)^k
 \le 2+\left(\frac{2r}{R}\right)^k.
\end{equation}
For $k\ne0,1$, equality holds if and only if $ABC$ is isosceles. For
$k=0$ and $k=1$, equality holds for every triangle.
\end{theorem}

\begin{proof}
Use the Ravi substitution
\[
 a=v+w,\qquad b=w+u,\qquad c=u+v,
\]
where $u=s-a$, $v=s-b$, $w=s-c$ are positive. Set
\[
 X=\frac{a(b+c-a)}{bc},\qquad
 Y=\frac{b(c+a-b)}{ca},\qquad
 Z=\frac{c(a+b-c)}{ab}.
\]
Then
\[
 X=\frac{2u(v+w)}{(u+v)(u+w)},\quad
 Y=\frac{2v(w+u)}{(v+w)(u+v)},\quad
 Z=\frac{2w(u+v)}{(u+w)(v+w)}.
\]
A direct multiplication gives
\begin{equation}\label{eq:prodXYZ}
 XYZ=\frac{8uvw}{(u+v)(v+w)(w+u)}.
\end{equation}
Also, direct simplification gives
\begin{equation}\label{eq:sumXYZ}
 X+Y+Z=2+\frac{8uvw}{(u+v)(v+w)(w+u)}.
\end{equation}
The expression on the right-hand side of \eqref{eq:prodXYZ} has a standard geometric form. Heron's formula gives
\[
 \Delta^2=(u+v+w)uvw,
\]
while
\[
 r=\frac{\Delta}{u+v+w},\qquad
 R=\frac{(u+v)(v+w)(w+u)}{4\Delta}.
\]
Therefore,
\begin{equation}\label{eq:q}
 \frac{2r}{R}=\frac{8uvw}{(u+v)(v+w)(w+u)}.
\end{equation}
Combining \eqref{eq:prodXYZ}--\eqref{eq:q},
\[
 X+Y+Z=2+q,\qquad XYZ=q,
 \qquad q:=\frac{2r}{R}.
\]
By Euler's inequality $R\ge2r$, we have $0<q\le1$. Applying Lemma~\ref{lem:power} to $X,Y,Z$ gives \eqref{eq:mainplus} and \eqref{eq:mainminus}.

We now determine the equality case. By Lemma~\ref{lem:power}, for
$k\ne0,1$ equality is equivalent to $(X,Y,Z)$ being a permutation of
$(q,1,1)$. Suppose $Y=Z=1$. From the formulas above,
\[
 Y=1\iff (v-w)(u-v)=0,
\]
and
\[
 Z=1\iff (w-v)(u-w)=0.
\]
If $v\ne w$, these two relations would force $u=v=w$, a contradiction. Hence $v=w$, equivalently $b=c$. The other permutations give $c=a$ or $a=b$. Conversely, if two sides are equal, two of $X,Y,Z$ equal $1$ and the third equals $XYZ=q$, so equality follows. Thus, for $k\ne0,1$, equality holds exactly for isosceles triangles.
For $k=0$ the two sides of \eqref{eq:mainminus} are equal to $3$, while
for $k=1$ equality follows from \eqref{eq:k1}.
\end{proof}

\end{document}